\documentclass[11pt]{amsart}

\usepackage{amsmath,amssymb}
\usepackage[margin=2.8cm]{geometry}
\DeclareFontShape{OT1}{cmss}{m}{n}{<->cmr10}{}

\newtheorem{theorem}{Theorem}
\newtheorem{proposition}[theorem]{Proposition}
\newtheorem{lemma}[theorem]{Lemma}

\newcommand{\R}{\mathbb{R}}

\begin{document}

\title[Equilateral excess and dimension cost]{Equilateral sets in $\ell_p^n$ beyond $n+1$:\ Linear excess and dimension cost}
\author{Logan R. Chalmers}
\address{University of Otago, Dunedin 9016, New Zealand}
\email{logan.chalmers@postgrad.otago.ac.nz}
\author{Aran S. Ziegler}
\address{Auckland University of Technology, Auckland 1010, New Zealand}
\email{pvx6961@autuni.ac.nz}
\date{}

\begin{abstract}
How many equilateral points can $\ell_p^n$ contain beyond $n+1$?
For every fixed real $p>4$, we construct equilateral sets for which
the number of points beyond $n+1$ grows at least linearly with $n$.
For every fixed
dimension, we prove that the maximum is $n+1$ at all sufficiently large
finite $p$, with an explicit threshold.  The least dimension admitting
any fixed number of extra points is $p^{1+o(1)}$ as $p\to\infty$.
The constructions and finite bounds are formalised in Lean.  Thus $p=4$
marks the boundary: $n+1$ is exact in every dimension for
$2\leq p\leq4$, while at every fixed finite $p>4$ the excess grows
at least linearly with dimension.
\end{abstract}

\maketitle

\section{Introduction}

How many points can an equilateral set in $\ell_p^n$ have as $p$ and
$n$ vary?  Write $e(\ell_p^n)$ for this maximum.  For every finite
$p>1$, one has $e(\ell_p^n)\geq n+1$.  In Guy's 1983 problem list,
Kusner conjectured that $e(\ell_p^n)=n+1$ for every $n\geq1$ and
$p>2$ \cite{Guy1983}.  Swanepoel established the case $p=4$;
Ge, Xu and Zhou later proved the full interval $2\leq p\leq4$
\cite{Swanepoel2004,GeXuZhou2026}.
Chalmers then disproved Kusner's conjecture as stated by Guy,
constructing $58$ equilateral points in $\ell_5^{56}$
\cite{Chalmers2026}.  The full $p>4$ range was subsequently covered
\cite{Xiong2026}.

The failure of the $n+1$ bound raises two questions.  At a fixed
$p>4$, how does the number of extra points grow with $n$?  At a fixed
dimension $n$, what happens as $p$ grows?

At every fixed $p>4$, the excess over $n+1$ grows at least linearly
with $n$.
\begin{theorem}[Linear excess]\label{thm:linear-excess}
 For every real $p>4$ there are constants $c_p>0$ and $n_p$ such that
 \[
 e(\ell_p^n)\geq(1+c_p)n \qquad\text{for every integer }n\geq n_p.
 \]
\end{theorem}

Theorem~\ref{thm:linear-excess} follows by joining copies of a
spherical seed.  If $S\subset\ell_p^d$ has $d+2$ points of common
norm-power $N$ and pairwise distance-power $D>2N$, then
Lemma~\ref{lem:join} gives $r(d+2)$ equilateral points in
$\ell_p^{r(d+1)-1}$ for every $r\geq1$.  The excess is therefore $r$.
Proposition~\ref{prop:seed} constructs such a seed for every $p>4$.

At each fixed dimension, sufficiently large finite $p$ forces the
maximum back to $n+1$.  At $p=\infty$, the $2^n$ vertices of the cube
are equilateral.
\begin{theorem}[High-exponent exclusion]\label{thm:exclusion}
 If $n\geq2$ is an integer and
 $p\geq\max\{6,2n(\log n)^2\}$ is real, then
 \[
 e(\ell_p^n)=n+1.
 \]
\end{theorem}

\newpage
The first two theorems bound the dimension needed to retain extra
points as $p$ grows.  For an integer $k\geq1$, define
\[
 D_k(p):=\min\bigl\{n\geq1:e(\ell_p^n)\geq n+1+k\bigr\}.
\]
The construction in Theorem~\ref{thm:linear-excess} makes $D_k(p)$
finite for $p>4$ and supplies an upper bound.  Theorem~\ref{thm:exclusion}
supplies a lower bound when $p$ is large.
\begin{theorem}[Dimension cost of excess]\label{thm:dimension}
 For every integer $k\geq1$ and real $p\geq6$,
 \[
 \frac{p}{2(\log p)^2}<D_k(p)\leq\frac{256}{9}kp.
 \]
 Consequently, $D_k(p)=p^{1+o(1)}$ as $p\to\infty$ for each fixed $k$.
\end{theorem}

Our seed has $32m$ points in $32m-2$ coordinates.  For $p\geq6$ we can
take $m\leq8p/9$.
Theorem~\ref{thm:exclusion} uses a rank obstruction for an approximate
distance matrix.
An earlier estimate gave
$e(\ell_p^n)\leq2(p+1)n$ for $p\geq c(n\log n)^2$ with an absolute
constant $c>0$ \cite{ChenGuiTangXiong2022}.

Swanepoel proved linear excess for $1<p<2$ \cite{Swanepoel2004}.
Together with Theorem~\ref{thm:linear-excess}, this gives linear excess
for every finite $p\in(1,2)\cup(4,\infty)$; for $2\leq p\leq4$, the
maximum is $n+1$.  The exact values of $e(\ell_p^n)$ for general pairs
$(p,n)$ and the exact first dimensions $D_k(p)$ remain open.

\section{Construction and linear excess}

For $x\in\R^d$ write $\|x\|_p^p=\sum_{j=1}^d|x_j|^p$.
Joining $r$ copies of a spherical seed with $32m$ points in $32m-2$
coordinates gives $32mr$ points in $r(32m-1)-1$ coordinates,
an excess of $r$ over dimension plus one.  The joining lemma is
related to \cite[Proposition~1]{Swanepoel2004}.

\begin{lemma}[Spherical amplification]\label{lem:join}
 Let $p>1$ and let $S\subset\ell_p^d$ have $M$ points, common
 distance-power $D>0$, and common norm-power $N$ with $2N<D$.
 For every integer $r\geq1$ there is an equilateral set of $rM$
 points in $\ell_p^{rd+r-1}$ with distance-power $D$.
\end{lemma}
\begin{proof}
 The case $r=1$ is $S$.  For $r\geq2$, let $b\in(0,1)$ solve
 \[
 (1+b)^p+(r-2)b^p=2.
 \]
 The left-hand side is strictly increasing, starts at $1$, and
 exceeds $2$ at $b=1$.  Thus the points
 \[
 e_1,\ldots,e_{r-1},-b(1,\ldots,1)\in\R^{r-1}
 \]
 have common distance-power $2$.  Scale these tags to distance-power
 $D-2N$.  Place $r$ copies of $S$ in disjoint $d$-coordinate blocks,
 appending the corresponding tag to each point of a copy.  Distances
 within a copy retain power $D$; cross-copy distances have power
 $N+N+(D-2N)=D$.
\end{proof}

To construct a seed for Lemma~\ref{lem:join}, define sixteen type
vectors from a Walsh matrix.
Index the rows and columns of the symmetric $8$-by-$8$ Walsh matrix
$H_8$ by $\mathbb F_2^3$, with
$(H_8)_{ij}=(-1)^{i\cdot j}$.  For $t>1$ define sixteen vectors in
$\R^{16}$ by
\[
 q_i=((H_8)_{i,*},t e_i),\qquad
 q_{8+i}=(t e_i,-(H_8)_{*,i}),\qquad i\in\mathbb F_2^3.
\]
Since $H_8$ is symmetric, either row or column notation may be used
in the second expression.  Put
\begin{align}
 R&=8+t^p,&
 A&=4\cdot2^p+2t^p,&
 B&=(t+1)^p+(t-1)^p+14, \label{eq:RAB}\\
 L&=2^{p-1}R,&
 T&=(1+m^{-1})L,&
 F_{p,m}(t)&=T+7A-8B. \label{eq:F}
\end{align}
Here and below $m$ is a power of two.  Each $q_i$ has norm-power
$R$.  Moreover, for distinct $i,j$ and any signs $a,b\in\{-1,1\}$,
\begin{equation}\label{eq:front}
 \|a q_i-b q_j\|_p^p=
 \begin{cases}
 A,&i,j\text{ belong to the same half of }\{0,\ldots,15\},\\
 B,&i,j\text{ belong to different halves.}
 \end{cases}
\end{equation}
Within one half, two signed Walsh rows or columns differ in exactly
four places, and the two spikes contribute $2t^p$.  Across halves,
fourteen coordinates contribute $1$ each; the two remaining terms
are $(t+1)^p$ and $(t-1)^p$.  This proves
\eqref{eq:front} for both sign choices.

\begin{proposition}[Spherical seed above four]\label{prop:seed}
 For every real $p>4$ there is a power of two $m$ and an
 equilateral set of $32m$ points in $\ell_p^{32m-2}$.
 Its points have common norm-power $N$, its pairwise distances have
 common power $D$, and $2N<D$.  For $p\geq6$ one may take $m$ to be
 the least power of two strictly greater than $4p/9$.
\end{proposition}
\begin{proof}
 We first give the coordinates for any $m$ and $t>1$ satisfying
 $F_{p,m}(t)=0$; the existence of these parameters is proved in
 Lemmas~\ref{lem:high-root} and \ref{lem:low-root} below.
 Set
 \[
 W=\frac{T-A}{8},\qquad
 \alpha^p=\frac{R}{16},\qquad
 \beta^p=\frac{mW}{2^p},
 \]
 choosing the positive roots.  These choices are valid because
 \[
 T-A=(2^{p-1}-2)t^p+\frac{L}{m}>0.
 \]
 The equation $F_{p,m}(t)=0$ gives
 \begin{equation}\label{eq:balance}
 A+8W=B+7W=T.
 \end{equation}

 Write $m=2^d$, index the rows and columns of $H_m$ by
 $\mathbb F_2^d$, and set
 $H_m(s,r)=(-1)^{s\cdot r}$.  Identify the type indices
 $0,\ldots,15$ with $\mathbb F_2^4$ and write
 $\chi_i(c)=(-1)^{i\cdot c}$.  Let
 $C=\mathbb F_2^4\setminus\{0,8\}$, where $8$ denotes the first
 coordinate character.  For each
 $(\sigma,i,s)\in\{-1,1\}\times\mathbb F_2^4\times\mathbb F_2^d$
 form a point $x_{\sigma,i,s}$ by concatenating
 \begin{align*}
 &\sigma H_m(s,r)q_i[a]
     &&(r\in\mathbb F_2^d,\ a=0,\ldots,15),\\
 &\alpha\chi_i(c)H_m(s,u)
     &&(c\in\mathbb F_2^4,\ u\in\mathbb F_2^d\setminus\{0\}),\\
 &\beta\chi_i(c)
     &&(c\in C).
 \end{align*}
 There are $32m$ labels and
 $16m+16(m-1)+14=32m-2$ coordinates.

 Put $D=2^pmR$.  If $i,s$ agree and only $\sigma$ changes, the
 front block contributes $D$.  If $i$ agrees and $s$ changes, the
 balance of distinct Walsh rows gives $2^{p-1}mR$ from the front
 block and the same amount from the ordinary character block.
 If $i\ne j$, the front contributes $mA$ or $mB$ according to
 the halves containing $i,j$.  The ordinary block contributes
 $2^{p-1}(m-1)R$: for each nonzero Walsh column, exactly eight of
 the sixteen character signs differ.  Of the fourteen retained
 characters, eight differ when $i,j$ are in the same half and
 seven differ when they are in different halves.  Thus the
 retained contribution is $8mW$ or $7mW$, respectively.
 Equation~\eqref{eq:balance} makes the total in either case
 \[
 mT+2^{p-1}(m-1)R=2^pmR=D.
 \]
 Since $D>0$, the labels give distinct points.

 Every point has norm-power
 \[
 N=mR+16(m-1)\alpha^p+14\beta^p
  =(2m-1)R+\frac{14mW}{2^p}.
 \]
 As $A>0$, one has $W<T/8$, so
 \[
 N<\frac{23m-1}{8}R
   <2^{p-1}mR=\frac D2
 \]
 for $p>4$.  The root lemmas complete the proof.
\end{proof}

\begin{lemma}[Scalar root for $p\geq6$]\label{lem:high-root}
 If $p\geq6$ and $m$ is the least power of two strictly greater
 than $4p/9$, the equation $F_{p,m}(t)=0$ has a root
 $t\in(1,64^{1/p})$, and $m\leq8p/9$.
\end{lemma}
\begin{proof}
 Direct substitution gives
 \[
 F_{p,m}(1)
 =49(2^{p-1}-2)+\frac{9\,2^{p-1}}{m}>0.
 \]
 Set $t_0=64^{1/p}$, $a=\log64$ and $y=a/(2p)$.
 With $F_{p,\infty}=L+7A-8B$, write
 $\varepsilon=-F_{p,\infty}(t_0)/L(t_0)$.
 Expanding \eqref{eq:RAB} at $t_0$ yields
 \[
 \varepsilon=
 \frac{16}{9}\bigl(\cosh(y)^p-1\bigr)
 +\frac29\left(\frac{t_0-1}{2}\right)^p
 -\frac{196}{9\cdot2^p}.
 \]
 Using $\cosh y\geq1+y^2/2+y^4/24$ and
 $(1+z)^p\geq1+pz+p(p-1)z^2/2$ for $z\geq0$ gives
 \begin{equation}\label{eq:epsilon}
 p\varepsilon\geq\frac{2a^2}{9}
 +\frac{a^4(3p-2)}{216p^2}
 -\frac{196p}{9\cdot2^p}.
 \end{equation}
 The bound $a>112/27$ follows from
 \[
 \log2=2\int_0^{1/3}\frac{dz}{1-z^2}
 >2\int_0^{1/3}(1+z^2)\,dz=\frac{56}{81}.
 \]
 On $6\leq p\leq8$, the middle term
 in \eqref{eq:epsilon} is bounded below by its value at $8$
 and the last term by its value at $6$, giving
 \[
 p\varepsilon>
 \frac{258663895}{114791256}>\frac94.
 \]
 For $p\geq8$, discarding the middle term and bounding the
 last term at $8$ gives
 \[
 p\varepsilon>
 \frac{164983}{52488}>\frac94.
 \]
 These comparisons use that $(3p-2)/p^2$ and $p/2^p$
 decrease for $p\geq6$.  Thus
 $1/m<9/(4p)<\varepsilon$, so
 $F_{p,m}(t_0)=L(t_0)(1/m-\varepsilon)<0$.
 Continuity supplies the root.  The least power of two above
 $4p/9$ is at most $8p/9$.
\end{proof}

\begin{lemma}[Exponential zero count]\label{lem:zero-count}
 If $0<a_1<\cdots<a_s$ and each $c_j\ne0$, then
 $f(x)=\sum_{j=1}^s c_ja_j^x$ has at most as many real zeros,
 counted with multiplicity, as the sign sequence
 $(c_1,\ldots,c_s)$ has changes.
\end{lemma}
\begin{proof}
 Induct on $s$.  Dividing $f$ by $a_1^x$ preserves its zeros.
 The derivative of the resulting function is an exponential
 polynomial with coefficient signs $(c_2,\ldots,c_s)$.
 By induction it has at most $V(c_2,\ldots,c_s)$ zeros.
 Rolle's theorem, including multiplicities, allows at most
 one additional zero for $f$.  When $c_1$ and $c_2$ have
 opposite signs, this is exactly the additional sign change.
 When they have the same sign, the normalized function and
 its derivative have that sign sufficiently far to the left.
 Any first zero of the function then requires a preceding
 zero of its derivative.  Together with the zeros supplied
 by Rolle between subsequent zeros, this removes the
 additional allowance.
\end{proof}

\begin{lemma}[Scalar root for $4<p<6$]\label{lem:low-root}
 For every real $4<p<6$ there are a power of two $m$ and
 $t\in(1,\sqrt8)$ satisfying $F_{p,m}(t)=0$.
\end{lemma}
\begin{proof}
 Put $t_0=\sqrt8$ and $f(p)=F_{p,\infty}(t_0)$.
 Algebraic expansion gives
 \[
 f(p)=-112-8(t_0-1)^p+32\,2^p+14t_0^p
       -8(t_0+1)^p+\tfrac12(2t_0)^p.
 \]
 Its bases are strictly increasing and its coefficient signs
 are $-,-,+,+,-,+$, with three changes.  Direct expansion
 gives, for general $t$,
 \[
 F_{2,\infty}(t)=0,\quad
 F_{4,\infty}(t)=6(t^2-8)^2,\quad
 F_{6,\infty}(t)=30(t^2-8)(t^4-8).
 \]
 Thus $f$ has zeros at $2,4,6$.  Lemma~\ref{lem:zero-count} shows
 that these are all its real zeros and each is simple.
 Its largest-base term is positive, hence $f(p)>0$ for
 $p>6$ and $f(p)<0$ for $4<p<6$.
 Choose a power of two $m>L(t_0)/[-f(p)]$.
 Then
 $F_{p,m}(t_0)=f(p)+L(t_0)/m<0$,
 whereas $F_{p,m}(1)>0$ by the preceding lemma's
 direct calculation.  Continuity gives the asserted root.
\end{proof}

\begin{proof}[Proof of Theorem~\ref{thm:linear-excess}]
 Fix $p>4$ and take the seed of Proposition~\ref{prop:seed}.
 Lemma~\ref{lem:join} gives, for every $r\geq1$,
 $32mr$ equilateral points in dimension
 $r(32m-1)-1$.  Appending zero coordinates preserves all distances,
 so for every $n\geq32m-2$,
 \begin{equation}\label{eq:floor}
 e(\ell_p^n)\geq
 32m\left\lfloor\frac{n+1}{32m-1}\right\rfloor.
 \end{equation}
 Consequently
 \[
 \liminf_{n\to\infty}\frac{e(\ell_p^n)}n
 \geq\frac{32m}{32m-1}>1.
 \]
 Choosing any $c_p<1/(32m-1)$ proves the theorem for all
 sufficiently large $n$.
\end{proof}

\section{High-exponent exclusion and dimension cost}

For Theorem~\ref{thm:exclusion}, a matrix of coordinate powers
approximates the distance matrix of a putative $(n+2)$-point set.
The error bound and rank constraint give a contradiction.

\begin{proof}[Proof of Theorem~\ref{thm:exclusion}]
 Suppose that $n+2$ equilateral points exist in $\ell_p^n$,
 with $n\geq2$ and
 $p\geq\max\{6,2n(\log n)^2\}$.  A larger set contains such a
 subset.  Rescale the common distance to $1$ and translate each
 coordinate by its minimum over the points.  Denote the resulting
 coordinates by $x_{i\ell}$, where $1\leq i\leq n+2$ and
 $1\leq\ell\leq n$.  Every $x_{i\ell}$ belongs to $[0,1]$, and
 \begin{equation}\label{eq:unit-distance}
 \sum_{\ell=1}^n|x_{i\ell}-x_{j\ell}|^p=1
 \qquad(i\ne j).
 \end{equation}
 Define the $(n+2)$-by-$n$ matrices
 \[
 P_{i\ell}=x_{i\ell}^p,\qquad
 Q_{i\ell}=(1-x_{i\ell})^p,
 \]
 and put $K=PQ^{\mathsf T}+QP^{\mathsf T}$ and
 $E=K-(J-I)$, where $J$ is the all-ones matrix.

 First, every entry of $E$ is nonnegative.  For $0\leq x\leq y\leq1$,
 set $d=y-x$ and $t=x(1-y)$.  Then
 $(1-x)y=d+t$, so
 \begin{equation}\label{eq:nonnegative}
 x^p(1-y)^p+(1-x)^py^p-d^p
 =(d+t)^p-d^p+t^p\geq0.
 \end{equation}
 Summing \eqref{eq:nonnegative} over coordinates and using
 \eqref{eq:unit-distance} proves $E_{ij}\geq0$ for $i\ne j$;
 the diagonal is nonnegative directly.

 We bound the off-diagonal error.  Put
 \[
 v=\frac{\log n}{p},\qquad
 u=\frac{(\log n)^2}{p},\qquad
 \rho=n^{-2/p}=e^{-2v}.
 \]
 The elementary inequality
 $t\leq(1-d)^2/4$ gives
 \begin{equation}\label{eq:cosh}
 d+t\leq\left(\frac{1+d}{2}\right)^2
 =d\cosh\!\left(\frac{\log(1/d)}2\right)^2
 \quad(d>0).
 \end{equation}
 Since $\log\cosh z\leq z^2/2$ for $z\geq0$,
 coordinates with $d\geq\rho$ satisfy
 $(d+t)^p-d^p\leq d^p(e^u-1)$.
 For $d<\rho$, the first expression in
 \eqref{eq:cosh} gives
 \[
 (d+t)^p\leq
 \left(\frac{1+\rho}{2}\right)^{2p}
 =n^{-2}\cosh(v)^{2p}\leq n^{-2}e^u.
 \]
 Also $t^p\leq4^{-p}$.  Summing the two classes of
 coordinates and using \eqref{eq:unit-distance}, we obtain
 \begin{equation}\label{eq:error-bounds}
 0\leq E_{ij}\leq
 \eta:=\left(1+\frac1n\right)e^u-1+n4^{-p}
 \quad(i\ne j).
 \end{equation}
 On the diagonal, a separate calculation gives
 \begin{equation}\label{eq:diagonal-bound}
 0\leq E_{ii}
 =2\sum_{\ell=1}^n[x_{i\ell}(1-x_{i\ell})]^p
 \leq\delta:=2n4^{-p}.
 \end{equation}

 There is a nonzero $c\in\R^{n+2}$ satisfying
 $P^{\mathsf T}c=0$ and $\sum_i c_i=0$, since these are
 at most $n+1$ homogeneous equations in $n+2$ unknowns.
 Hence $c^{\mathsf T}Kc=0$ and
 $c^{\mathsf T}(J-I)c=-\|c\|_2^2$, so
 \begin{equation}\label{eq:error-contradiction}
 c^{\mathsf T}Ec=\|c\|_2^2.
 \end{equation}
 Let $s$ be the sum of the positive entries of $c$; the
 negative entries sum to $-s$.  Nonnegativity of $E$
 allows us to discard the terms with $c_ic_j<0$ in an
 upper bound.  The same-sign off-diagonal coefficients
 sum to $2s^2-\|c\|_2^2$.  Thus
 \[
 c^{\mathsf T}Ec
 \leq\eta(2s^2-\|c\|_2^2)+\delta\|c\|_2^2
 \leq\left(\frac n2\eta+\delta\right)\|c\|_2^2,
 \]
 where the last step uses
 $4s^2=\|c\|_1^2\leq(n+2)\|c\|_2^2$.

 The coefficient is strictly less than $1$.  For $n\geq3$,
 $u\leq1/(2n)$ and $e^u\leq(1-u)^{-1}
 \leq2n/(2n-1)$.  Since $\log n\geq\log3>1$,
 the hypothesis implies $p\geq2n$.  Therefore
 \[
 \frac n2\eta+\delta
 \leq\frac{3n}{2(2n-1)}
   +\left(\frac{n^2}{2}+2n\right)16^{-n}
 \leq\frac9{10}+\frac{21}{8192}<1.
 \]
 The last exponential expression decreases for integer $n\geq3$.
 For $n=2$, the bounds $\log2<3/4$ and $p\geq6$ give
 $u<1/8$ and $e^u\leq8/7$, whence
 \[
 \frac n2\eta+\delta\leq
 \frac57+\frac6{4096}<1.
 \]
 This contradicts \eqref{eq:error-contradiction}.

 Finally, $n+1$ points always exist.  Take
 $e_1,\ldots,e_n,-a(1,\ldots,1)$, where $a\in(0,1)$
 solves $(1+a)^p+(n-1)a^p=2$.
 Their pairwise distance-powers all equal $2$.
 Thus $e(\ell_p^n)=n+1$.
\end{proof}

\begin{proof}[Proof of Theorem~\ref{thm:dimension}]
 Fix $p\geq6$ and $k\geq1$.  Let $m$ be the seed order
 of Lemma~\ref{lem:high-root}.  Joining $k$ copies of the seed yields
 $32mk$ points in dimension
 $n=k(32m-1)-1$.  Since $32mk=n+1+k$,
 \[
 D_k(p)\leq k(32m-1)-1
 \leq k\left(\frac{256}{9}p-1\right)-1
 <\frac{256}{9}kp.
 \]
 Conversely, let $n=D_k(p)$.  Then $n\geq2$ and
 $\ell_p^n$ contains at least $n+2$ equilateral points.
 Theorem~\ref{thm:exclusion} implies
 $p<2n(\log n)^2$.  If $n\leq p$, this gives
 $p<2n(\log p)^2$.  If $n>p$, the desired lower bound
 holds because $2(\log p)^2>1$.  In either case,
 \[
 \frac{p}{2(\log p)^2}<D_k(p).
 \]
 For fixed $k$, taking logarithms of the two bounds gives
 $\log D_k(p)/\log p\to1$, which is
 $D_k(p)=p^{1+o(1)}$.
\end{proof}

For each fixed $p>4$, the number of equilateral points beyond $n+1$
grows at least linearly with dimension.  For each fixed dimension,
sufficiently large $p$ restores the maximum $n+1$.  Our bounds place
the dimension needed to retain any fixed number of extra points at
$p^{1+o(1)}$.  The remaining question is whether that dimension must
grow linearly with $p$.

\medskip
\noindent\textbf{Question.} For each fixed integer $k\geq1$, is there
a constant $c_k>0$ such that $D_k(p)\geq c_kp$ for all sufficiently
large real $p$?

\section*{Data access statement}
The Lean formalisation is available on Zenodo
(DOI: \texttt{10.5281/zenodo.22917956}).

\end{document}